\documentclass[11pt,twoside,reqno,centertags,draft]{amsart}

\PII{}
\copyrightinfo{}{}

\usepackage{amsmath,amsthm,amsfonts,amssymb}
\usepackage[mathscr]{euscript}
\usepackage{hyperref}
\usepackage{mathrsfs}
\usepackage{graphicx}
\usepackage{color}
\usepackage{epsfig}

\newtheorem{thm}{Theorem}
\newtheorem{propo}{Proposition}

\newtheorem{rmq}{Remark}
\newtheorem{lemma}{Lemma}

\def\dis{\displaystyle}
\def\Om{\Omega}
\def\om{\omega}

\newcommand{\eps}{\varepsilon}
\newcommand{\Fin}{\hfill$\Box$}

\newcommand{\N}{\mbox{$I \kern -4pt N$}}

\newcommand{\Q}{\mbox{$Q \kern -8pt I$}}
\newcommand{\R}{\mbox{$I \kern -4pt R$}}
\newcommand{\C}{\mbox{$C \kern -8pt I$}}

\newcommand{\mat}[1]{\mbox{\boldmath{$#1$}}}
\begin{document}

\title{ On the control of the Burgers-alpha model }
\footnotetext[1]{Partially supported by INCTMat, CAPES and CNPq (Brazil) by grants 307893/2011-1, 552758/2011-6 and 477124/2012-7.}
\footnotetext[2]{Partially supported by grant MTM2010-15592 (DGI-MICINN, Spain).}
\footnotetext[3]{Partially supported by CAPES (Brazil) and grants MTM2010-15592 (DGI-MICINN, Spain).}
\date{}
\maketitle

\vspace{ -1\baselineskip}

{\small
\begin{center}
 {\sc F\'{a}gner D. Araruna\footnotemark[1]} \\
Departamento de Matem\'{a}tica, Universidade Federal da Para\'iba, 58051-900, Jo\~{a}o Pessoa--PB, Brasil  \\

\vspace{0.5\baselineskip}

 {\sc Enrique Fern\'{a}ndez-Cara\footnotemark[2]} \\
Dpto.\ EDAN, University of Sevilla, Aptdo.~1160, 41080~Sevilla, Spain.  \\

\vspace{0.5\baselineskip}

 {\sc Diego A. Souza\footnotemark[3]} \\
Dpto.\ EDAN, University of Sevilla, 41080~Sevilla, Spain\\

\vspace{0.5\baselineskip}

\end{center}
}

\numberwithin{equation}{section}
\allowdisplaybreaks

\smallskip

 \begin{quote}
\footnotesize
{\bf Abstract.}
  	This work is devoted to prove the local null controllability of the Burgers-$\alpha$ model.
   	The state is the solution to a regularized Burgers equation, where the transport term is of the form $zy_x$,
  	$z=(Id-\alpha^2\frac{\partial^2}{\partial x^2})^{-1}y$ and $\alpha>0$ is a small parameter.
   	We also prove some  results concerning the behavior of the null controls and associated states  as $\alpha\to 0^+$.

\end{quote}


\section{Introduction and main results}

	Let $L>0$ and $T>0$ be positive real numbers. {Let $(a,b)\subset (0,L)$ be a (small) nonempty open subset which will be referred as the control domain.}

	We will consider the following controlled system for the Burgers equation:
\begin{equation}\label{CP}
	\left\{
		\begin{array}{lll}
			y_t - y_{xx} + y y_x = v1_{(a,b)}    		  & \text{in} &  (0,L)\times(0,T),   \\
			y(0,\cdot) = y(L,\cdot) = 0             	   	  & \text{on}&  (0,T),                    \\
			y(\cdot,0) = y_0                      			  & \text{in} &  (0,L).
		\end{array}
	\right.
\end{equation}
	
	In~\eqref{CP}, the function $y=y(x,t)$ can be interpreted as a one-dimensional velocity of a fluid and $y_{0}=y_{0}(x)$ is an initial datum.
	The function $v=v(x,t)$ (usually in~$L^{2}((a,b)\times(0,T))$) is the control acting on the system and $1_{(a,b)}$ denotes the
	characteristic function of $(a,b)$.
	
	In this paper, we will also consider a system similar to \eqref{CP}, where the transport term is of the form $zy_x$, where
	$z$ is the solution to an elliptic problem governed by $y$. Namely, we consider the following {\it regularized} version of~\eqref{CP},
	where $\alpha > 0$:
\begin{equation}\label{CPalpha}
	\left\{
		\begin{array}{lll}
			y_t - y_{xx} + z y_x = v1_{(a,b)}                   		 	   & \text{in}& (0,L)\times(0,T),   \\
			z - \alpha^2 z_{xx} = y                                			   & \text{in}& (0,L)\times(0,T),   \\
			y(0,\cdot) = y(L,\cdot) = z(0,\cdot) = z(L,\cdot) = 0  		   & \text{on}& (0,T),                   \\
			y(\cdot,0) = y_0                                      				   & \text{in}& (0,L).
		\end{array}
	\right.
\end{equation}

	This will be called in this paper the Burgers-$\alpha$ system. It is a particular case of the systems introduced in~\cite{Holm-Staley}
	to describe the balance of convection and stretching in the dynamics of one-dimensional nonlinear waves in a fluid with small viscosity.
	It can also be viewed as a simplified 1D~version of the so called Leray-$\alpha$ system, introduced to describe turbulent flows as an alternative
	to the classical averaged Reynolds models, see~\cite{F-HOLM-T}; see also~\cite{Holm}. By considering a special kernel associated to the
	Green's function for the Helmholtz operator, this model compares successfully with empirical data from turbulent channel and pipe flows
	for a wide range of Reynolds numbers, at least for periodic boundary conditions, see~\cite{Holm} (the Leray-$\alpha$ system is also closely related to the systems treated by
	Leray in~\cite{Leray} to prove the existence of solutions to the Navier-Stokes equations; see~\cite{HOLM-M-R}).
	
	Other references concerning systems of the kind~\eqref{CPalpha} in one and several dimensions are~\cite{Ch-HOLM, G-HOLM} and~\cite{SHEN-G-T, TIAN-S-D},
	respectively for numerical and optimal control issues.
	
	Let us present the notations used along this work. The symbols $C$, $\hat C$ and $C_{i}$, $i = 0, 1, \dots$ stand for generic positive constants
	(usually depending on $a$, $b$, $L$ and~$T$). For any $r\in [ 1,\infty]$ and any given Banach space $X$, $\|\cdot\|_{L^{r}(X)}$
	will denote the usual norm in $L^{r}(0,T;X)$. In particular, the norms in $L^{r}(0,L)$ and~$L^{r}((0,L)\times(0,T))$ will be denoted by $\|\cdot\|_{r}$.
	We will also need the Hilbert space $K^2(0,L) := H^2(0,L) \cap H_0^1(0,L)$.

	The null controllability problems for~\eqref{CP} and~\eqref{CPalpha} at time $T>0$ are the following:
\begin{quote}\it{
	For any $y_0\in H^{1}_{0}(0,L)$, find $v\in L^2((a,b)\times(0,T))$ such that the associated solution to~\eqref{CP}
	$($resp.~\eqref{CPalpha}$)$ satisfies}
\end{quote}
\begin{equation}\label{null_condition}
	y(\cdot,T) = 0\quad\textit{in}\quad(0,L).
\end{equation}

	Recently, important progress has been made in the controllability ana\-lysis of linear and semilinear parabolic equations and systems.
	We refer to the works \cite{AD-FC-MB-Z, FPZ, FC-Z, F-I-2, Z0, EZ-rev}.
	In particular, the controllability of the Burgers equation has been analyzed in~\cite{CHAP, DIAZ, C-G1, F-I-2, SGR-Oleg, HORSIN}.
	Consequently, it is natural to try to extend the known results to systems like \eqref{CPalpha}.
	Notice that \eqref{CPalpha} is different from \eqref{CP} at least in two aspects:
	first, the occurrence of non-local in space nonlinearities;
	secondly, the fact that a small paramter $\alpha$ appears.
	
	Our first main results are the following:
	
\begin{thm}\label{NC-Burgers}
	For each $T>0$, the system \eqref{CPalpha} is locally null-controllable at time $T$. More precisely, there exists $\delta>0$
	$($independent of $\alpha$$)$ such that, for any $y_0\in H^1_0(0,L)$ with $\|y_0\|_{\infty} \leq \delta$, there exist controls
	$v_\alpha\in L^\infty((a,b)\times(0,T))$ and associated states $(y_\alpha,z_\alpha)$ satisfying \eqref{null_condition}. Moreover, one has
\begin{equation}\label{v-unif}
	\|v_\alpha\|_\infty\leq C \quad \forall \alpha>0.
\end{equation}
\end{thm}

	{
\begin{thm}\label{Burgers-LargeT}
	For each $y_0\in H^1_0(0,L)$ with $\|y_0\|_\infty < \pi/L$, the system \eqref{CPalpha} is null-controllable at large time.
	In other words, there exist $T>0$
	(independent of $\alpha$), controls $v_\alpha\in L^\infty((a,b)\times(0,T))$ and associated states $(y_\alpha,z_\alpha)$
	satisfying \eqref{null_condition} and~\eqref{v-unif}.
\end{thm}
	}

	{Recall that $\pi/L$ is the square root of the first eigenvalue of the Dirichlet Laplacian in this case.
	On the other hand, notice that these results provide controls in~$L^\infty((a,b)\times(0,T))$ and not only in~$L^2((a,b)\times(0,T))$}.
	In fact, this is very convenient not only in~\eqref{CP}  and~\eqref{CPalpha}, but also in some intermediate problems arising in the proofs, since this way we obtain better estimates for the states and the existence and convergence assertions are easier to establish.
	
	The main novelty of these results is that they ensure the control of a kind of nonlocal nonlinear parabolic equations.
	This makes the difference with respect to other previous works, such as~\cite{FPZ} or~\cite{AD-FC-MB-Z, FC-Z}.
	This is not frequent in the analysis of the controllability of PDEs. Indeed, in general when we deal with nonlocal nonlinearities,
	it does not seem easy to transmit the information furnished by locally supported controls to the whole domain in a satisfactory way.

	We will also prove a result concerning the controllability in the limit, as~$\alpha\to 0^+$.
	More precisely, the following holds:
\begin{thm}\label{conver-cont}
	Let $T>0$ be given and let $\delta>0$ be the constant furnished by Theorem~$\ref{NC-Burgers}$.
	Assume that $y_0\in H^1_0(0,L)$ with $\|y_0\|_{L^\infty} \leq \delta$, let $v_\alpha$ be a null control
	for~\eqref{CPalpha} satisfying \eqref{v-unif} and let $(y_\alpha,z_\alpha)$ be an associated state satisfying \eqref{null_condition}.
	Then, at least for a subsequence, one has
\begin{equation}\label{conv}
	\begin{array}{c}
		v_\alpha\to v\hbox{ weakly-$*$} \hbox{ in } L^\infty((a,b)\times(0,T)),\\
		z_\alpha \to y \hbox{ and } y_\alpha \to y\hbox{ weakly-$*$} \hbox{ in } L^\infty((0,L)\times(0,T))
	\end{array}
\end{equation}
as $\alpha\to0^+$, where $(v,y)$ is a control-state pair  for~\eqref{CP} that verifies \eqref{null_condition}.
\end{thm}

	The rest of this paper is organized as follows.
	In~Section~\ref{Sec2}, we prove some results concerning the existence, uniqueness and regularity of the solution to \eqref{Palpha}.
	Sections~\ref{Sec3}, \ref{Sec4}, and~\ref{Sec5} deal with the proofs of Theorems~\ref{NC-Burgers}, \ref{Burgers-LargeT} and~\ref{conver-cont}, respectively.
	Finally, in~Section~\ref{Sec6}, we present some additional comments and questions.


\section{Preliminaries}\label{Sec2}

 	In this Section, we will first establish a result concerning global existence and uniqueness for the
	Burgers-$\alpha$ system
\begin{equation}\label{Palpha}
	\left\{
		\begin{array}{lll}
			y_t - y_{xx} + z y_x = f                                			& \text{in} & (0,L)\times(0,T),        \\
			z - \alpha^2 z_{xx} = y                               			& \text{in} & (0,L)\times(0,T),        \\
			y(0,\cdot) = y(L,\cdot) = z(0,\cdot) = z(L,\cdot) = 0   		& \text{on}& (0,T),                         \\
			y(\cdot,0) = y_0                                        				& \text{in} & (0,L).
		\end{array}
	\right.
\end{equation}

	It is the following:
\begin{propo}\label{E-U-B-alpha}
	Assume that $\alpha>0$.
	Then, for any $f\in L^\infty((0,L)\times(0,T))$ and $y_0\in H^1_0(0,L)$, there exists exactly one solution $(y_\alpha,z_\alpha)$ to \eqref{Palpha}, with
\[
	\begin{array}{c}
		\noalign{\smallskip}\dis
		y_\alpha\in L^2(0, T;H^2(0,L)) \cap C^0([0, T];H^1_0(0,L)),\\
		\noalign{\smallskip}\dis
		z_\alpha\in L^2\big(0,T;H^4(0,L)\big) \cap L^\infty\big(0,T;H^1_0(0,L)\cap H^3(0,L)\big),\\
		\noalign{\smallskip}\dis
		(y_\alpha)_t\in L^2((0,L)\times(0,T)),~(z_\alpha)_t\in L^2\big(0,T;H^2(0,L)\big).
	\end{array}
\]
	Furthermore, the following estimates hold:
\begin{equation}\label{yalpha-ineq}
	\begin{alignedat}{2}
		\noalign{\smallskip}\dis
		\|(y_\alpha)_t\|_{2}+\|y_\alpha\|_{L^2(H^2)}+\|y_\alpha\|_{L^\infty(H_0^1)}\leq&~C(\|y_0\|_{H_0^1}+\|f\|_{2})e^{C(M(T))^2},\\
		\noalign{\smallskip}\dis
		\|z_\alpha\|^2_{L^\infty(L^2)}+2\alpha^2\|z_\alpha\|^2_{L^\infty(H^1_0)}\leq&~\| y_\alpha\|^2_{L^\infty( L^2)},\\
		\noalign{\smallskip}\dis
		2\alpha^2\|(z_\alpha)_x\|^2_{L^\infty(L^2)}+\alpha^4\|(z_\alpha)_{xx}\|^2_{L^\infty(L^2)}\leq&~\| y_\alpha\|^2_{L^\infty( L^2)},\\
		\noalign{\smallskip}\dis
		\|y_\alpha\|_\infty\leq&~ M(T),\\
		\noalign{\smallskip}\dis
		\|z_\alpha\|_\infty\leq&~ M(T),
	\end{alignedat}
\end{equation}
	where $M(t):=\|y_0\|_\infty+t\|f\|_\infty$.
\end{propo}
\begin{proof}

	\textsc{a) Existence:}
	We will reduce the proof to the search of a fixed point of an appropriate mapping $\Lambda_\alpha$.
	
	Thus, for each $\overline y\in L^{\infty}((0,L)\times(0,T))$, let $z = z(x,t)$ be the unique solution to
\begin{equation}\label{elipt-pde}
	\left\{
		\begin{array}{lll}
			z-\alpha^2  z_{xx} = \overline y,    	 &\text{in}& (0,L)\times(0,T) ,    \\
			z(0,\cdot) =  z(L,\cdot) = 0        		 &\text{on}& (0,T).
		\end{array}
	\right.
\end{equation}
	Since $\overline y\in L^{\infty}((0,L)\times(0,T))$, it is clear that $z \in L^\infty(0,T; K^2(0,L))$.
	Then, thanks to the Sobolev embedding, we have $z,~z_x\in L^{\infty}((0,L)\times(0,T))$ and the following is satisfied:
\begin{equation}\label{z-ineq}
	\begin{alignedat}{2}
		\|z\|^2_{L^\infty(L^2)}+2\alpha^2\|z\|^2_{L^\infty(H^1_0)}\leq&~\|\overline y\|^2_{L^\infty( L^2)},\\
		2\alpha^2\|z_x\|^2_{L^\infty(L^2)}+\alpha^4\|z_{xx}\|^2_{L^\infty(L^2)}\leq &~\|\overline y\|^2_{L^\infty( L^2)},\\
		\|z\|_{\infty}\leq &~\|\overline y\|_{\infty}.
	\end{alignedat}
\end{equation}
	From this $z$, we can obtain $y$ as the unique solution to the linear problem
\begin{equation}\label{y-eq}
	\left\{
		\begin{array}{lll}
			y_t  - y_{xx} + z y_x = f     	 & \text{in}&  (0,L)\times(0,T),  \\
			y(0,\cdot) = y(L,\cdot) = 0   	 &\text{on}&  (0,T),                   \\
			y(\cdot,0) = y_0              		 &\text{in} &  (0,L).
		\end{array}
	\right.
\end{equation}
	Since $z, f\in L^{\infty}((0,L)\times(0,T))$ and $y_0\in H^1_0(0,L)$, it is {\it clear} that
\[
	\begin{array}{c}
		\noalign{\smallskip}\dis
		y \in L^2(0, T;K^2(0,L)) \cap C^0([0, T];H^1_0(0,L)),\\
		\noalign{\smallskip}\dis
		y_t \in L^2((0,L) \times (0,T))
	\end{array}
\]
	and we have the following estimate:
\begin{equation}\label{y-ineq}
	\|y_t\|_{2}+\|y\|_{L^2(H^2)}+\|y\|_{L^\infty(H_0^1)}\leq C(\|y_0\|_{H_0^1}+\|f\|_{2})e^{C\|z\|_\infty^2}.
\end{equation}
	Indeed, this can be easily deduced, for instance, from a standard Galerkin approximation and Gronwall's Lemma;
	see for instance~\cite{D-JLL}.

	We will use the following result, whose proof is given below, after the proof of this Theorem.

\begin{lemma}\label{est-L-infty}
	The solution $y$ to~\eqref{y-eq} satisfies
\begin{equation}\label{y-infty}
	\|y\|_{\infty}\leq M(T).
\end{equation}
\end{lemma}

	Now, we introduce the Banach space
\begin{equation}\label{spaceW}
	W = \{w\in L^\infty(0,T; H_0^1(0,L)) : w_t\in L^2((0,L)\times(0,T))\},
\end{equation}
	the closed ball
\[
	K=\{w\in L^{\infty}((0,L)\times(0,T)) : \|w\|_{\infty}\leq  M(T)\}
\]
	and the mapping $\tilde\Lambda_\alpha$, with $\tilde\Lambda_\alpha(\overline y)=y$ for all $\overline y \in L^{\infty}((0,L)\times(0,T))$.
	Obviously $\tilde\Lambda_\alpha$ is well defined and, in view of Lemma~\ref{est-L-infty}, maps the whole space
	$L^{\infty}((0,L)\times(0,T))$ into $W \cap K$.
		
	Let us denote by $\Lambda_\alpha$ the restriction to~$K$ of~$\tilde\Lambda_\alpha$.
	Then, thanks to Lemma~\ref{est-L-infty}, $\Lambda_\alpha$ maps $K$ into itself.
	Moreover, it is clear that $\Lambda_\alpha: K \mapsto K$ satisfies the hypotheses of Schauder's Theorem.
	Indeed, this nonlinear mapping is continuous and compact
	(the latter is a consequence of the fact that, if $B$ is bounded in $L^{\infty}((0,L)\times(0,T))$, then $\Lambda_\alpha(B)$
	is bounded in $W$ and therefore it is relatively compact in the space $L^{\infty}((0,L)\times(0,T))$, in view of the
	classical results of the Aubin-Lions' kind, see for instance~\cite{Simon}). Consequently, $\Lambda_\alpha$ possesses at least one fixed point in $K$.
	
	This immediately achieves the proof of existence.

\

	\textsc{b) Uniqueness:}
	Let $(z_\alpha',y_\alpha')$ be another solution to~\eqref{Palpha} and let us introduce $u:=y_\alpha-y_\alpha'$ and $m:=z_\alpha-z_\alpha'$.
	Then
\[
	\left\{
		\begin{array}{lll}
			u_t - u_{xx} + z_\alpha u_x = - m (y_\alpha')_x        		& \text{in}& (0,L)\times(0,T),   \\
			m - \alpha^2 m_{xx} = u                                			& \text{in}& (0,L)\times(0,T),   \\
			u(0,\cdot) = u(L,\cdot) = m(0,\cdot) = m(L,\cdot) = 0  	& \text{on}& (0,T),                   \\
			u(\cdot,0) = 0                                               			& \text{in}& (0,L).
		\end{array}
	\right.
\]
	Since $y_\alpha'\in L^2(0,T;H^2(0,L))$, thanks to the Sobolev embedding, we have $y_\alpha'\in  L^2(0,T;C^1[0,L])$.
	Therefore, we easily get from the first equation of the previous system that
\[
	\frac{1}{2}\frac{d}{dt}\|u\|_2^2 +\|u_x\|_2^2 \leq \|z_\alpha\|_\infty\|u_x\|_2\|u\|_2+\|(y_\alpha')_x\|_\infty\|m\|_2\|u\|_2.
\]
	Since $\|m\|_2\leq\|u\|_2$, we have
\[
	\frac{d}{dt}\|u\|_2^2+\|u_x\|_2^2 \leq\bigg(\|z_\alpha\|^2_\infty+2\|(y_\alpha')_x\|_\infty\bigg)\|u\|_2^2.
\]	
	Therefore, in view of Gronwall's Lemma, we necessarily have $u\equiv0$.
	Accor\-dingly, we also obtain $m\equiv0$ and uniqueness holds.
\end{proof}
	
	Let us now return to Lemma \ref{est-L-infty} and establish its proof.

\noindent
	\begin{proof}[Proof of Lemma~\ref{est-L-infty}]
	Let $y$ be the solution to~\eqref{y-eq} and let us set $w= (y-M(t))_+$.
	Notice that $w(x,0) \equiv 0$ and $w(0,t) \equiv w(L,t) \equiv 0$.
	
	Let us multiply the first equation of~\eqref{y-eq} by $w$ and let us integrate on $(0,L)$.
	Then we obtain the following for all $t$:
\[
	\int_0^L (y_tw +zy_xw)\,dx +\int_0^L y_xw_x \,dx = \int_0^Lfw \,dx.
\]
	This can also be written in the form
\[
	\int_0^L (w_tw +zw_xw)\,dx +\int_0^L |w_x|^2 \,dx = \int_0^L(f-M_t)w \,dx
\]
	and, consequently, we obtain the identity
\[
	\frac{1}{2}\frac{d}{dt}\|w\|_2^2+\|w_x\|_2^2 -\frac{1}{2}\int_0^L z_x|w|^2\,dx= \int_0^L(f-\|f\|_\infty)w\,dx
\]
	and, therefore,
\begin{equation}\label{f1}
	\frac{1}{2}\frac{d}{dt}\|w\|_2^2+\|w_x\|_2^2 -\frac{1}{2}\int_0^L z_x|w|^2\,dx\leq 0.
\end{equation}
	Since $z_{x}\in L^{\infty}((0,L)\times(0,T))$, it follows by (\ref{f1}) that
\[
	\frac{d}{dt}\|w\|_2^2 \leq \|z_x\|_\infty\|w\|_2^2.
\]
	Then, using again Gronwall's Lemma, we see that $w\equiv0$.
	
	Analogously, if we introduce $\widetilde w=  (y+M(t))_-$, similar computations lead to the identity $\widetilde w\equiv0$.
	Therefore, $y$ satisfies \eqref{y-infty} and the Lemma is proved.
	
\end{proof}
		
	We will now see that, when $f$ is fixed and $\alpha \to 0^+$, the solution to~\eqref{Palpha} converges to the solution to the Burgers system
\begin{equation}\label{P}
	\left\{
		\begin{array}{lll}
			y_t - y_{xx} + y y_x = f              		& \text{in}&  (0,L)\times(0,T),   \\
			y(0,\cdot) = y(L,\cdot) = 0        	        & \text{on}&  (0,T),                   \\
			y(\cdot,0) = y_0       		                & \text{in}&  (0,L).
		\end{array}
	\right.
\end{equation}
	
\begin{propo}\label{convergence}
	Assume that $y_0\in H^1_0(0,L)$ and $f\in L^\infty((0,L)\times(0,T))$ are given. For each $\alpha>0$, let $(y_\alpha,z_\alpha)$ be
	the unique solution to~\eqref{Palpha}. Then
\begin{equation}\label{cccconvzy}
	z_\alpha \to y \hbox{ and } y_\alpha\to y \hbox{ strongly in } L^2(0,T;H_0^1(0,L))
\end{equation}
	as $\alpha\to 0^+$, where $y$ is the unique solution to~\eqref{P}.
\end{propo}

\begin{proof}
	Since $(y_\alpha,z_\alpha)$ is the solution to~\eqref{Palpha}, we have \eqref{yalpha-ineq}.
	Therefore, there exists $y$ such that, at least for a subsequence, we have
\begin{equation}\label{convery}
	\begin{array}{c}
		y_\alpha \to y \hbox{ weakly in } L^2(0, T;H^2(0,L)),					 \\
		\noalign{\smallskip}
		y_\alpha \to y  \hbox{ weakly-}* \hbox{ in } L^\infty(0,T;H_0^1(0,L)),		 \\
		\noalign{\smallskip}
		(y_\alpha)_t \to y_t\hbox{ weakly in } L^2((0,L)\times(0,T)).
	\end{array}
\end{equation}

	The Hilbert space
\[
	Y=\{\, w\in L^2(0, T; K^2(0,L)) : w_t\in L^2((0,L)\times(0,T)) \,\}
\]
	is compactly embedded in $L^2(0, T;H^1_0(0,L))$.	Consequently,
\begin{equation}\label{conv-y-L2H01}
	y_\alpha \to y \hbox{ strongly in } L^2(0, T;H^1_0(0,L)).
\end{equation}

	Let us see that $y$ is the unique solution to~\eqref{P}.
	
	Using the second equation in \eqref{Palpha}, we have
\[
	(z_\alpha-y)-\alpha^2(z_\alpha-y)_{xx}=(y_\alpha-y)+\alpha^2y_{xx}.
\]
	Multiplying this equation by $-(z_\alpha-y)_{xx}$ and integrating in $(0,L)\times(0,T)$, we obtain
\[
	\begin{alignedat}{2}
		\int_0^T\!\!\!\int_0^L |(z_\alpha-y)_x|^2\,dx\,dt~+~& \alpha^2\int_0^T\!\!\!\int_0^L |(z_\alpha-y)_{xx}|^2\,dx\,dt \\
	 	&= \int_0^T\!\!\!\int_0^L (y_\alpha-y)_x(z_\alpha-y)_x\,dx\,dt \\
	 	&-\alpha^2\int_0^T\!\!\!\int_0^L y_{xx}(z_\alpha-y)_{xx}\,dx\,dt,
	\end{alignedat}
\]
	whence
\[
	\int_0^T\!\!\!\int_0^L |(z_\alpha-y)_x|^2\,dx\,dt\leq\int_0^T\!\!\!\int_0^L |(y_\alpha-y)_x|^2\,dx\,dt+\alpha^2\| y_{xx}\|^2_2.
\]	
	This shows that
\begin{equation}\label{conv-z-L2}
	z_\alpha \rightarrow y\hbox{ strongly in } L^2(0, T;H^1_0(0,L))
\end{equation}
 	and, consequently,
\begin{equation}\label{convprod}
	z_\alpha(y_\alpha)_x\to yy_x,\hbox{ strongly in } L^1((0,L)\times(0,T)).
\end{equation}

	Finally, for each $\psi\in L^\infty(0,T;H^1_0(0,L))$, we have
\begin{equation}\label{alphapsi}
	\int_0^T\!\!\!\int_0^L  \left((y_\alpha)_t\psi+(y_\alpha)_x\psi_x+z_\alpha(y_\alpha)_x\psi\right) \,dx\,dt=\int_0^T\!\!\!\int_0^L f\psi\,dx\,dt.
\end{equation}
	Using \eqref{convery} and \eqref{convprod}, we can take limits in all terms and find that
\begin{equation}\label{psi}
	\int_0^T\!\!\!\int_0^L  \left(y_t\psi+y_x\psi_x+yy_x\psi\right)\,dx\,dt=\int_0^T\!\!\!\int_0^L f\psi\,dx\,dt,
\end{equation}
	that is, $y$ is the unique solution to \eqref{P}.
	
	This proves that \eqref{cccconvzy} holds at least for a subsequence.
	But, in view of uniqueness, not only a subsequence but the whole sequence converges.
\end{proof}

\begin{rmq}{\rm
	In fact, a result similar to Proposition~\ref{convergence} can also be established with varying $f$ and $y_0$.
   	More precisely, if we introduce data $f_\alpha$ and $(y_{0})_{\alpha}$ with
\[
	f_\alpha \to f \hbox{ weakly-$*$ in } L^\infty((0,L) \times (0,T))
\]
	and
\[
	(y_{0})_{\alpha}\to y_0 \hbox{ weakly-$*$ in } L^\infty(0,L),
\]
	then we find that the associated solutions $(y_\alpha,z_\alpha)$ satisfy again \eqref{cccconvzy}.
		\Fin}
\end{rmq}

	To end this Section, we will now recall a result dealing with the null controllability of general parabolic linear systems of the form
\begin{equation}\label{parab}
	\left\{
		\begin{array}{lll}
		     y_t - y_{xx} + Ay_x = v1_{(a,b)}        		    & \text{in}  &  (0,L)\times(0,T),  \\
		     y(0,\cdot) = y(L,\cdot) = 0               		    & \text{on} &  (0,T),                   \\
		     y(\cdot,0) = y_0                            	  	    & \text{in}  &  (0,L).
		\end{array}
	\right.
\end{equation}
	where $y_0\in L^2(0,L)$, $A\in L^\infty((0,L)\times(0,T))$ and $v\in L^2((a,b)\times(0,T))$.

	It is well known that there exists exactly one solution $y$ to \eqref{parab}, with
\[
y\in C^0([0, T];L^2(0,L))\cap L^2(0, T;H^1_0(0,L)).
\]
	
	Related to controllabilty result, we have the following:

\begin{thm}\label{NC-Parab}
	The linear system \eqref{parab} is null controllable at any time $T>0$.
	In other words, for each $y_0\in L^2(0,L)$ there exists $v \in L^2((a,b)\times(0,T))$ such that the associated solution to~\eqref{parab} satisfies~\eqref{null_condition}.
	Furthermore, the extremal problem
\begin{equation}\label{optm}
	\!\!\!\!\!
	\left\{
		\begin{array}{l}
			\noalign{\smallskip}\!\! \text{Minimize } \ \displaystyle{\frac{1}{2}\int_0^T\!\!\!\int_a^b|v|^2 \,dx\,dt}\\
			\noalign{\smallskip}\!\! \text{Subject to:  $v \in L^2((a,b)\times(0,T))$, \eqref{parab}, \eqref{null_condition}}
		\end{array}
	\right.
\end{equation}
possesses exactly one solution $\hat v$ satisfying
\begin{equation}\label{est-cont}
	\|\hat v\|_2\leq C_0\|y_0\|_2,
\end{equation}
	where
\[
	C_0=e^{C_1(1+1/T+(1+T)\|A\|^2_\infty)}
\]
	and $C_1$ only depends on $a$, $b$ and $L$.
\end{thm}
		
	The proof of this result can be found in~\cite{ImanuvilovYamamoto}.


\section[Controllability of the Burgers-$\alpha$ mo\-del]{Local null controllability of the Burgers-$\alpha$ mo\-del}\label{Sec3}

	In this Section, we present the proof of Theorem~\ref{NC-Burgers}.

	Roughly speaking, we fix $\overline y$, we solve~\eqref{elipt-pde}, we control exactly to zero the linear system \eqref{parab} with $A = z$  and we set  $\Lambda_\alpha(\overline{y}) = y$.
	Then the task is to solve the fixed point  equation $y = \Lambda_\alpha(y)$.
	
	Several fixed point theorems can be applied.
	In this paper, we have preferred to use Schauder's Theorem, although other results also lead to the good conclusion;
	for instance, an argument relying on Kakutani's Theorem, like in~\cite{AD-FC-MB-Z}, is possible.
	
	As mentioned above, in order to get good properties for $\Lambda_\alpha$, it is very appropriate
	that the control belongs to $L^\infty$. This can be achieved by several ways;
	for instance, using an ``improved'' observability estimate for the solutions to the adjoint of~\eqref{parab} and arguing as in~\cite{AD-FC-MB-Z}.
	We have preferred here to use other techniques that rely on the regularity of the states and were originally used in~\cite{MBurgos-0};
	see also~\cite{MBurgos}.

	Let $y_0 \in H_0^1(0,L)$ and $a'$, $a''$, $b'$ and~$b''$ be given, with $0 < a < a'< a'' < b'' < b' < b < L$.
	Let $\theta$ and $\eta$ satisfy
\[
	\theta\in C^\infty([0,T]),~\theta \equiv 1~\text{in}~[0,T/4],~\theta\equiv 0~\text{in}~[3T/4,T],
\]
\[
	\eta\in\mathcal{D}(a,b),~\eta\equiv 1~\text{in a neighborhood of}~[a',b'],~0\leq\eta\leq 1.
\]

	As in the proof of Proposition~\ref{E-U-B-alpha}, we can associate to each $\overline y\in L^{\infty}((0,L)\times(0,T))$ the function $z$ through~\eqref{elipt-pde}.
	Recall that $z, z_x\in L^{\infty}((0,L)\times(0,T))$ and the inequalities \eqref{z-ineq} are satisfied. 	
	In view of Theorem~\ref{NC-Parab}, we can associate to $z$ the null control $\hat v$ of minimal norm in~$L^2((a'',b'') \times (0,T))$, that is,
	the solution to~\eqref{parab}--\eqref{optm} with $a$, $b$ and $A$ respectively replaced by $a''$, $b''$ and $z$.
	Let us denote by $\hat y$ the corresponding solution to~\eqref{parab}.

	Then, we can write that $\hat y=\theta(t)\hat u + \hat w$, where $\hat u$ and $\hat w$ are the unique solutions to the linear systems
\begin{equation}\label{syst u}
	\left\{
		\begin{array}{lll}
			 \hat u_t  - \hat u_{xx} +z\hat  u_x = 0    	   & \text{in}&  (0,L)\times(0,T),  \\
			 \hat u(0,\cdot) = \hat u(L,\cdot) = 0     	   & \text{on}&  (0,T), 		 \\
			 \hat u(\cdot,0) = y_0                        		   & \text{in}&  (0,L)
		\end{array}
	\right.
\end{equation}
and
\begin{equation}\label{syst tildew}
	\left\{
		\begin{array}{lll}
			 \hat  w_t  -  \hat  w_{xx} +z \hat w_x = \hat v1_{( a'', b'')}-\theta_t\hat u                & \text{in}&  (0,L)\times(0,T),  \\
			 \hat  w(0,\cdot) =  \hat  w(L,\cdot) = 0                                                  		     & \text{on}&  (0,T),                  \\
			 \hat  w(\cdot,0) = 0,~\hat w(\cdot,T) = 0                                               		     & \text{in}&  (0,L),
		\end{array}
	\right.
\end{equation}
	respectively.
	
	If we now set $w:= (1-\eta(x)) \hat w$, then we have that $w$ is the unique solution of the parabolic system
\begin{equation}\label{syst w}
	\left\{
		\begin{array}{lll}
	 		 w_t  - w_{xx} +z w_x =v-\theta_t\hat u             & \text{in}&  (0,L)\times(0,T),  \\
	 		 w(0,\cdot) = w(L,\cdot) = 0                               & \text{on}&  (0,T), \\
	 		 w(\cdot,0) = 0,~ w(\cdot,T) = 0                         & \text{in}&  (0,L),
		\end{array}
	\right.
\end{equation}	
	where $v:=\eta\theta_t\hat u-\eta_xz\hat  w+2\eta_x\hat  w_x+\eta_{xx}\hat  w+(1-\eta(x))\hat  v1_{(a'',b'')}$.

	Notice that $(1-\eta)\hat  v1_{(a'',b'')}\equiv 0$, since $\eta \equiv 1$ in $[a',b']$.
	Therefore, one has
\begin{equation}\label{controlinfty}
	v=\eta\theta_t\hat u-\eta_xz\hat w+2\eta_x\hat w_x+\eta_{xx}\hat w
\end{equation}
	and	then $supp~v\subset (a,b)$.
		
	Let us prove that $v \in L^\infty((a,b)\times(0,T))$ and
\begin{equation}\label{ineq}
	\|v\|_\infty\leq \hat C \|y_0\|_\infty,
\end{equation}
	for some
\begin{equation}\label{ineqp}
	\hat C= e^{C(a,b,L)(1+1/T+(1+T)\|\overline y\|^2_\infty)}.
\end{equation}
	
	First, note that $\hat u\in L^\infty((0,L)\times(0,T))$ and $\|\hat u\|_\infty\leq \|y_0\|_\infty$. Defining
\[
	G=(a,a^{\prime})\cup(b^{\prime},b),
\]
	we see that it suffices to check that $\eta_xz\hat w$, $\eta_x\hat{w}_{x}$ and~$\eta_{xx}\hat{w}$ belong to~$L^{\infty}(G\times(0,T))$,
	with norms in~$L^{\infty}(G\times(0,T))$ bounded by a constant times the $L^2$-norm of $\hat{v}$ and the $L^\infty$-norm of~$y_{0}$,
	since $\eta_x$ and $\eta_{xx}$ are identically zero in a neighborhood of $[a',b']$.
	
	From the usual parabolic estimates for \eqref{syst tildew} and~the estimate \eqref{z-ineq}, we first obtain that
\begin{equation}\label{global wtilde}
	\|\hat w_t\|_{L^2(L^2)}+\|\hat w\|_{L^2(H^2)}+\|\hat w\|_{L^\infty(H^1_0)}\leq
	\|\hat v1_{(a'', b'')}-\theta_t\hat u\|_{L^2(L^2)}e^{C\|\overline y\|^2_\infty}.
\end{equation}	
	In particular, we have $\hat w \in L^\infty((a,b)\times(0,T))$, with appropriate estimates.
	
	{
	On the other hand, $\theta_t \hat u \in L^\infty((0,L)\times(0,T))$ and, from the equation satisfied by $\hat w$,  we have
\[
	\hat  w_t  -  \hat  w_{xx} +z \hat w_x = -\theta_t\hat u \ \ \text{in} \ \ [(0,a'')\cup (b'',L)]\times(0,T).
\]
	
	Hence, from standard
	(local in space) parabolic estimates, we deduce that $\hat w$ belongs to the space~$X^p(0,T;G)=\{\hat w \in L^p(0,T;W^{2,p}(G)):\hat w_t\in L^p(0,T;L^p(G))\}$
	for all $2<p<+\infty$.

	Then, using Lemma~3.3 (p.~80) of~\cite{Ladyzhenskaja}, we can take $p > 3$ to get the embedding
	$X^p(0,T;G) \hookrightarrow C^0([0,T];C^1(\overline G))$ and~$\hat w_x \in C^0(\overline G\times [0,T])$.
	} This proves that $\hat w_x \in L^\infty(G)$, again with the appropriate estimates.
	
	
	Therefore, if we define $y:=\theta(t)\hat u + w$, one has
\begin{equation}\label{Burgers null-infty}
	\left\{
		\begin{array}{lll}
			 y_t  - y_{xx} +z y_x = v1_{( a, b)} 	 & \text{in} &  (0,L)\times(0,T),  	\\
			 y(0,\cdot) = y(L,\cdot) = 0         	 & \text{on}&  (0,T), 				\\
			 y(\cdot,0) = y_0                   		 & \text{in} &  (0,L),
		\end{array}
	\right.
\end{equation}	
	and~\eqref{null_condition}.
	Moreover, the control $v$ satisfies~\eqref{ineq}--\eqref{ineqp}.

	Let us set $\Lambda_\alpha(\overline y) = y$.
	In this way, we have been able to introduce a mapping
\[
	\Lambda_\alpha : L^{\infty}((0,L)\times(0,T)) \mapsto L^{\infty}((0,L)\times(0,T))
\]
	for which the following properties are easy to check:

\begin{itemize}

\item[a)]
	$\Lambda_\alpha$ is continuous and compact.
	{
	The compactness can be explained as follows:
	if $B \subset L^{\infty}((0,L)\times(0,T))$ is bounded, then $\Lambda_\alpha(B)$ is bounded in the space $W$ in~\eqref{spaceW} and, therefore, it is relatively compact in~$L^{\infty}((0,L)\times(0,T))$, in view of classical results of the Aubin-Lions' kind, see for instance~\cite{Simon}).
	}	
	
\item[b)]
	If $R>0$ and $\|y_0\|_{\infty}\leq \varepsilon(R)$ (independent of $\alpha$!), then $\Lambda_\alpha$ maps the ball
	$B_R:=\{\, \overline y\in L^\infty((0,L)\times(0,T)) : \|\overline y\|_{\infty}\leq R \,\}$ into itself.
	
\end{itemize}	
	
	The consequence is that, again, Schauder's Theorem can be applied and there exists controls $v_\alpha\in L^\infty((0,L)\times(0,T))$ such that
	the corresponding solutions to~\eqref{CPalpha} satisfy \eqref{null_condition}.	This achieves the proof of Theorem~\ref{NC-Burgers}.
	

\section[Large time null controllability]{Large time null controllability of the Burgers-$\alpha$ system} \label{Sec4}

	The proof of Theorem~\ref{Burgers-LargeT} is similar.
	It suffices to replace the assumption ``$y_0$ is small'' by an assumption imposing that $T$ is large enough.
	Again, this makes it possible to apply a fixed point argument.
	
{
	More precisely, let us accept that, if $y_0\in H^1_0(0,L)$ and~$\|y_0\|_\infty < \pi/L$, then the associted uncontrolled solution $y_\alpha$ to~\eqref{CPalpha} satisfies
\begin{equation}\label{decay0}
	\|y_\alpha(\cdot\,,t)\|_{H_0^1} \leq C(y_0) e^{-{1\over2}((\pi/L)^2 - \|y_0\|^2_\infty)t}
\end{equation}
	where $C(y_0)$ is a constant only depending on~$\|y_0\|_\infty$ and~$\|y_0\|_{H_0^1}$.
	Then, if we first take $v \equiv 0$, the state $y_\alpha(\cdot\,,t)$ becomes small for large $t$.
	In a second step, when $\|y_\alpha(\cdot\,,t)\|_{H_0^1}$ is sufficiently small, we can apply Theorem~\ref{NC-Burgers} and drive the state exactly to zero.

	Let us now see that \eqref{decay0} holds.
	Arguing as in the proof of Proposition~\ref{E-U-B-alpha}, we see that
\begin{equation}\label{eq1}
	\frac{d}{dt}\|y_\alpha\|_2^2+\|(y_\alpha)_x\|^2_2\leq \|y_0\|^2_\infty\|y_\alpha\|^2_2
\end{equation}
	 and, using Poincar\'{e}'s inequality, we obtain:
\[
	\frac{d}{dt}\|y_\alpha\|^2_2+ (\pi/L)^2 \|y_\alpha\|^2_2\leq \|y_0\|^2_\infty\|y_\alpha\|^2_2.
\]
	Let us introduce $r={1\over2}((\pi/L)^2 - \|y_0\|^2_\infty)$.
	It then follows that
\begin{equation}\label{eq2}
	\|y_\alpha(\cdot\,,t)\|_2^2 \leq \|y_0\|^2_2 e^{-2rt}.
\end{equation}	
	Hence, by combining  \eqref{eq1} and \eqref{eq2}, it is easy to see that
\[
	\frac{d}{dt} \left(e^{rt}\|y_\alpha\|^2_2 \right)+e^{rt}\|(y_\alpha)_x\|^2_2\leq (r+\|y_0\|^2_\infty)\|y_0\|^2_2 \, e^{-rt}.
\]	
	Integrating from $0$ to $t$ yields	
\begin{equation}\label{ultima}
	\int_0^te^{r\sigma}\|(y_\alpha)_x\|^2_2 \, d\sigma \leq \bigg(2+\frac{\|y_0\|^2_\infty}{r}\bigg)\|y_0\|^2_2.
\end{equation}
	Now, we take the $L^2$-inner product of \eqref{Palpha} and~$-(y_\alpha)_{xx}$ and get
\[
	\frac{d}{dt}\|(y_\alpha)_x\|_2^2\leq \|y_0\|^2_\infty\|(y_\alpha)_x\|^2.
\]	
	Multiplying this inequality by $e^{rt}$, we deduce that
\[
	\frac{d}{dt}\left(e^{rt}\|(y_\alpha)_x\|^2_2\right) \leq (r+\|y_0\|^2_\infty)\,e^{rt}\,\|(y_\alpha)_x\|^2_2
\]	
and, consequently, we see from~\eqref{ultima} that
\[
	\|(y_\alpha)_x(\cdot\,,t)\|^2_2 \leq \bigg[(r+\|y_0\|^2_\infty)\bigg(2+\frac{\|y_0\|^2_\infty}{r}\bigg)\|y_0\|^2_2+\|y_0\|^2_{H_0^1}\bigg]e^{-rt},
\]
	which implies \eqref{decay0}.
	
\begin{rmq}{\rm
	To our knowledge, it is unknown what can be said when the smallness assumption $\|y_0\|_\infty < \pi/L$ is not satisfied.
	In fact, it is not clear whether or not the solutions to~\eqref{CPalpha} with large initial data and $v \equiv 0$ decay as $t \to +\infty$.
	\Fin}
\end{rmq}

}	


\section{Controllability in the limit}\label{Sec5}

	In this Section, we are going to prove Theorem~\ref{conver-cont}.
	
	For the null controls $v_\alpha$ furnished by Theorem~\ref{NC-Burgers} and the associated solutions $(y_\alpha,z_\alpha)$ to~\eqref{CPalpha},
	we have the uniform estimates \eqref{ineq} and \eqref{yalpha-ineq} with $f = v_\alpha 1_{(a,b)}$. Then, there exists $y\in L^2(0, T;K^2(0,L))$, with
	$y_t\in L^2((0,L)\times(0,T))$, and $v\in L^\infty((a,b)\times(0,T))$ such that, at least for a subsequence, one has:
\begin{equation}\label{converycontrol}
	\begin{array}{l}
		y_\alpha  \to y \hbox{ weakly in } L^2(0, T;K^2(0,L)),\\
		\noalign{\smallskip}
		(y_\alpha)_t \to y_t \hbox{ weakly in } L^2((0,L)\times(0,T))\\
		\noalign{\smallskip}
		v_\alpha  \to v \hbox{ weakly}-* \hbox{ in } L^\infty((a,b)\times(0,T)).
	\end{array}
\end{equation}
	As before, the Aubin-Lions' Lemma implies that
\begin{equation}\label{conv-y-L2H01-control}
	y_\alpha \to y \hbox{ strongly in } L^2(0, T;H^1_0(0,L)).
\end{equation}
	Using the second equation in \eqref{CPalpha}, we see that
\[
	(z_\alpha-y)-\alpha^2(z_\alpha-y)_{xx}=(y_\alpha-y)+\alpha^2y_{xx}.
\]
	Multiplying this equation by $-(z_\alpha-y)_{xx}$ and integrating in $(0,L)\times(0,T)$, we deduce
\[
	\begin{alignedat}{2}
		\int_0^T\!\!\!\int_0^L |(z_\alpha-y)_x|^2\,dx\,dt &+ \alpha^2\int_0^T\!\!\!\int_0^L |(z_\alpha-y)_{xx}|^2\,dx\,dt \\
		\noalign{\smallskip}
		&= \int_0^T\!\!\!\int_0^L (y_\alpha-y)_x(z_\alpha-y)_x\,dx\,dt \\
		\noalign{\smallskip}
		&-\alpha^2\int_0^T\!\!\!\int_0^L y_{xx}(z_\alpha-y)_{xx}\,dx\,dt.
	\end{alignedat}
\]
	Whence,
\[
	\int_0^T\!\!\!\int_0^L |(z_\alpha-y)_x|^2\,dx\,dt\leq\int_0^T\!\!\!\int_0^L |(y_\alpha-y)_x|^2\,dx\,dt+\alpha^2\| y_{xx}\|^2_2.
\]	
	This shows that
\begin{equation}\label{conv-z-L2-control}
	z_\alpha \rightarrow y\hbox{ strongly in } L^2(0, T;H^1_0(0,L)).
\end{equation}
 	and the transport terms in \eqref{CPalpha} satisfy
\begin{equation}\label{convprodcontrol}
	z_\alpha(y_\alpha)_x \to yy_x \hbox{ strongly in } L^1((0,L)\times(0,T)).
\end{equation}
	In this way, for each $\psi\in L^\infty(0,T;H^1_0(0,L))$, we obtain
\begin{equation}\label{alpha-psi-cont}
	\int_0^T\!\!\!\int_0^L  \left((y_\alpha)_t\psi+(y_\alpha)_x\psi_x+z_\alpha(y_\alpha)_x\psi\right)\,dx\,dt=\int_0^T\!\!\!\int_0^Lv_\alpha1_{(a,b)}\psi\,dx\,dt.
\end{equation}
	Using \eqref{converycontrol} and \eqref{convprodcontrol}, we can pass to the limit, as $\alpha\to0^+$, in all the terms of
	\eqref{alpha-psi-cont} to find
\begin{equation}\label{psi-cont}
	\int_0^T\!\!\!\int_0^L  \left(y_t\psi+y_x\psi_x+yy_x\psi\right)\,dx\,dt=\int_0^T\!\!\!\int_0^L v1_{(a,b)}\psi\,dx\,dt,
\end{equation}
	that is, $y$ is the unique solution of \eqref{CP} and $y$ satisfies \eqref{null_condition}.


\section{Additional comments and questions}\label{Sec6}


\subsection{A boundary controllability result}

	We can use an extension argument to prove local boundary controllability results similar to those above.
	
	For instance, let us see that the analog of Theorem~\ref{NC-Burgers}  remains true.
	Thus, let us introduce the controlled system
\begin{equation}\label{BCPalpha}
	\left\{
		\begin{array}{lll}
			y_t - y_{xx} + z y_x = 0                                         		& \text{in}  &  (0,L)\times(0,T),  \\
			z - \alpha^2 z_{xx} = y                                          		& \text{in} 	 &  (0,L)\times(0,T),  \\
			z(0,\cdot)  = y(0,\cdot) =  0,~z(L,\cdot)= y(L,\cdot) =u      & \text{on} &  (0,T),             \\
			y(\cdot,0) = y_0                                                 		& \text{in}  &  (0,L),
		\end{array}
	\right.
\end{equation}
	where $u = u(t)$ stands for the control function and $y_0\in H_0^1(0,L)$ is given.

	Let $a$, $b$ and $\tilde L$ be given, with $L<a<b<\tilde L$. Then, let us define $\tilde y_0 : [0,\tilde L] \mapsto \mathbb{R}$, with $\tilde y_0 := y_01_{[0,L]}$.
	Arguing as in~Theorem~\ref{NC-Burgers}, it can be proved that there exists $(\tilde y, \tilde v)$, with $\tilde v \in L^\infty((a,b) \times (0,T))$,
\[
\left\{
\begin{array}{lll}
	\tilde y_t - \tilde y_{xx} + z 1_{[0,L]} \,\tilde y_x = \tilde v 1_{(a,b)}            &\text{in}&(0,\tilde L)\times(0,T),\\
	 z - \alpha^2 z_{xx} = \tilde y                                                       &\text{in}&(0, L)\times(0,T),          \\
	\tilde y(0,\cdot) = z(0,\cdot) =\tilde y(\tilde L,\cdot) = 0  &\text{on}&(0,T),                      \\
	 z( L,\cdot) =    \tilde y( L,\cdot)         &\text{on}&(0,T),                      \\
	\tilde y(\cdot,0) =\tilde  y_0                                                        &\text{in}&(0,\tilde L),           \\
\end{array}
\right.
\]
and~$\tilde y(x,T) \equiv 0$.
	Then, $y:=\tilde y 1_{(0,L)}$, $z$ and $u(t):= \tilde y(L,t)$ satisfy \eqref{BCPalpha}.
	
	{
	Notice that the control that we have obtained satisfies $u \in C^0([0,T])$, since it can be viewed as the lateral trace of a strong solution of the heat equation with a $L^\infty$ right hand side.
	}
	

\subsection{No global null controllability?}

	To our knowledge, it is unknown whether a general global null controllability result holds for \eqref{CPalpha}.
	We can prove global null controllability ``for large $\alpha$''.
	
	More precisely,  the following holds:

\begin{thm}\label{NC-large-alpha}
	Let $y_0\in H^1_0(0,L)$ and $T>0$ be given.
	There exists $\alpha_0=\alpha_0(y_0,T)$ such that \eqref{CPalpha} can be controlled to zero for all $\alpha>\alpha_0$.
\end{thm}

\begin{proof}	
	The main idea is, again, to apply a fixed point argument in \linebreak$L^\infty(0,T; L^2(0,L))$. 	
	
	For each $\overline y\in L^\infty(0,T; L^2(0,L))$, we introduce the solution $z$ to~\eqref{elipt-pde}. We notice that $z$ satisfies
\[
	\begin{alignedat}{2}
		\|z\|^2_{2}+2\alpha^2\|z_x\|^2_{2}\leq&~\|\overline y\|^2_{2},\\
		2\alpha^2\|z_x\|^2_{2}+\alpha^4\|z_{xx}\|^2_{2}\leq&~\|\overline y\|^2_{2}.
	\end{alignedat}
\]
	Then, as in the proof of Theorem~\ref{NC-Burgers}, we consider the solution $(y,v)$ to the system
\begin{equation}\label{BBBBBurgers null-infty}
	\left\{
		\begin{array}{lll}
			 y_t  - y_{xx} +z y_x = v1_{( a, b)}             & \text{in} &  (0,L)\times(0,T),  \\
			 y(0,\cdot) = y(L,\cdot) = 0                     	& \text{on}&  (0,T), 			\\
			 y(\cdot,0) = y_0                                 	& \text{in} &  (0,L),
		\end{array}
	\right.
\end{equation}	
	where we assume that $y$ satisfies \eqref{null_condition} and $v$ satisfies the estimate
\begin{equation}\label{ineqqq}
	\|v\|_\infty\leq \hat C\|y_0\|_\infty,\\
\end{equation}
	with
	$$
	\hat C= e^{C(a,b,L)(1+1/T+(1+T)\|z\|^2_\infty)}.
	$$
	It is then clear that
\begin{equation}\label{yalpha-ineqqqq}
\begin{alignedat}{2}
	\|y_t\|_{2}+\|y\|_{L^2(H^2)}+\|y\|_{L^\infty(H_0^1)}\leq&~C\|y_0\|_{H_0^1}e^{C(a,b,L)(1+1/T+(1+T)\|z\|^2_\infty)}.\\
\end{alignedat}
\end{equation}
	Since $\|z\|^2_{\infty}\leq \frac{C}{\alpha^2}\|\overline y\|^2_2$, we have
\[
	\|y_t\|_{2}+\|y\|_{L^2(H^2)}+\|y\|_{L^\infty(H_0^1)}\leq C\|y_0\|_{H_0^1}e^{C(a,b,L)\big(1+1/T+(1+T)\frac{1}{\alpha^2}\|\overline y\|^2_{L^\infty(L^2)}\big)}.
\]
	We can check that there exist $R$ and $\alpha_0$ such that
	$$
	C\|y_0\|_{H_0^1}e^{C(a,b,L)\big(1+1/T+(1+T)\frac{1}{\alpha^2}R^2\big)}<R,
	$$
for all $\alpha>\alpha_0$.
	Therefore, we can apply the fixed point argument in the ball $B_R$ of $L^2((0,L)\times(0,T))$ for these $\alpha$. This ends the proof.
\end{proof}

\

	Notice that we cannot expect \eqref{CPalpha} to be globally null-controllable with controls bounded independently of~$\alpha$, since the limit problem \eqref{CP} is not globally null-controllable, see~\cite{C-G1, SGR-Oleg}.
	More precisely, let $y_0\in H^1_0(0,L)$ and $T>0$ be given and let us denote by $\hat\alpha(y_0,T)$ the infimum of all $\alpha_0$ furnished by Theorem~\ref{NC-large-alpha}.
	Then, either $\hat\alpha(y_0,T) > 0$ or the associated cost of null controllability grows to infinity as $\alpha \to 0$, i.e.~the null controls of minimal norm $v_\alpha$ satisfy
	$$
\limsup_{\alpha \to 0^+} \|v_\alpha\|_{L^\infty((a,b) \times (0,T))} = +\infty.
	$$


\subsection{The situation in higher spatial dimensions. The Leray-$\alpha$ system}

       Let $\Om \subset \mathbb{R}^N$ be a bounded connected and regular open set
       ($N = 2$ or $N = 3$) and let $\om \subset \Om$ be a
       (small) open set.
       We will use the notation $Q := \Om \times (0,T)$ and~$\Sigma := \partial\Om \times (0,T)$ and we will use bold symbols for  vector-valued functions
       and spaces of vector-valued functions.

       For any $\mathbf{f}$ and~any $\mathbf{y}_0$ in appropriate spaces, we will consider the Navier-Stokes system
\begin{equation}\label{NS2}
	\left\{
		\begin{array}{lll}
			\mathbf{y}_t  - \Delta \mathbf{y} +(\mathbf{y}\cdot \nabla) \mathbf{y}+ \nabla p =  \mathbf{f}          	 & \text{in} &       Q,    \\
			\nabla \cdot \mathbf{y} =0                                                                             					 & \text{in} &       Q,    \\
			\mathbf{y} = \textbf{0}                                                                                 					 & \text{on}&  \Sigma, \\
			\mathbf{y}(0) = \mathbf{y}_0                                                                           					 & \text{in} &  \Omega.
		\end{array}
	\right.
\end{equation}
	As before, we will also introduce a smoothing kernel and a related modification of~\eqref{NS2}.
	More precisely, the following so called Leray-$\alpha$ model will be of interest:
\begin{equation}\label{Lalpha}
	\left\{
		\begin{array}{lll}
			\mathbf{y}_t  - \Delta \mathbf{y} +(\mathbf{z} \cdot \nabla) \mathbf{y}+ \nabla p = \mathbf{f}          & \text{in} &       Q,     \\
			\nabla \cdot \mathbf{y}=\nabla \cdot \mathbf{z}= 0                                                     			& \text{in} &       Q,     \\
			\mathbf{z}-\alpha^2\Delta \mathbf{z} +\nabla \pi=\mathbf{y}                                            			& \text{in} &       Q,     \\
			\mathbf{y} = \mathbf{z}=\textbf{0}                                                                     					& \text{on}&  \Sigma,  \\
			\mathbf{y}(0) = \mathbf{y}_0                                                                           					& \text{in} &  \Omega.
		\end{array}
	\right.
\end{equation}	

	Let us recall the definitions of some function spaces that are frequently used in the analysis of incompressible fluids:
\begin{equation*}
\begin{array}{c}
	\mathbf{H}= \dis \left\{\, \mat{\varphi}\in\mathbf{L}^2(\Omega) : \nabla \cdot \mat{\varphi} = 0~\text{in}~ \Omega,
	\ \ \mat{\varphi}\cdot		\mathbf{n} = 0~\text{on}~\partial\Omega\,\right\},\\
	\noalign{\smallskip}
	\mathbf{V}=\dis\left\{\,\mat{\varphi}\in\mathbf{H}^1_0(\Omega) : \nabla\cdot\mat{\varphi}=0~ \text{in}~\Omega\,\right\}.
\end{array}
\end{equation*}

	It is not difficult to prove that, for any $\alpha > 0$, under some reasonable conditions on~$\mathbf{f}$ and~$\mathbf{y}_0$,
	\eqref{Lalpha} possesses a unique global weak solution.
	{
	This is stated rigo\-rously in the following proposition, that we present without proof
	(the arguments are similar to those in~\cite{TEMAM}; the detailed proof will appear in a forthcoming paper):
	}
	
\begin{propo}\label{E-U-L-alpha}
	Assume that $\alpha>0$.
	Then, for any $\mathbf{f}\in L^2(0,T;\mathbf{H}^{-1}(\Om))$ and any $\mathbf{y}_0\in \mathbf{H}$, there exists exactly one
	solution $(\mathbf{y}_\alpha,p_\alpha,\mathbf{z}_\alpha,\pi_{\alpha})$ to~\eqref{Lalpha}, with
\[
	\mathbf{y}_\alpha\in L^2(0, T;\mathbf{V}) \cap C^0([0, T];\mathbf{H}),~(\mathbf{y}_\alpha)_t\in L^1(0,T;\mathbf{V}'),
\]
\[
	\mathbf{z}_\alpha\in L^2\big(0,T;\mathbf{H}^2(\Om)\cap \mathbf{V}\big) \cap L^\infty\big(0,T;\mathbf{H}\big).
\]
	Furthermore, the following estimates hold:
\begin{equation}\label{yalpha-ineqq}
	\begin{alignedat}{2}
		\|(\mathbf{y}_\alpha)_t\|_{L^1(\mathbf{V}')}+\|\mathbf{y}_\alpha\|_{L^2(\mathbf{V})}+\|\mathbf{y}_\alpha\|_{L^\infty(\mathbf{H})}
		\leq&~C(\|\mathbf{y}_0\|_{2}+\|\mathbf{f}\|_{L^2(\mathbf{H}^{-1})}),														\\
		\|\mathbf{z}_\alpha\|^2_{L^\infty(\mathbf{H})}+2\alpha^2\|\mathbf{z}_\alpha\|^2_{L^\infty(\mathbf{V})}
		\leq&~\| \mathbf{y}_\alpha\|^2_{L^\infty( \mathbf{H})},																\\
		2\alpha^2\|(\mathbf{z}_\alpha)_x\|^2_{L^\infty(\mathbf{H})}+\alpha^4\|\Delta(\mathbf{z}_\alpha)\|^2_{L^\infty(\mathbf{H})}
		\leq&~\| \mathbf{y}_\alpha\|^2_{L^\infty( \mathbf{H})}.
	\end{alignedat}
\end{equation}
\end{propo}

	In view of the estimates \eqref{yalpha-ineqq}, there exists $\mathbf{y}\in L^2\left(0, T;\mathbf{V})\right)$ with~$\mathbf{y}_t\in L^1(0,T;\mathbf{V'})$
	such that, at least for a subsequence,
   \begin{equation}\label{converyvec}
\left.
\begin{array}{l}
	\mathbf{y}_\alpha \to \mathbf{y} \hbox{ weakly in } L^2\left(0, T;\mathbf{V})\right), \\
\noalign{\smallskip}	(\mathbf{y}_\alpha)_t \to  \mathbf{y}_t \hbox{ weakly-* in } L^1(0,T;\mathbf{V'}).
\end{array}
\right.
   \end{equation}

	Thanks to the Aubin-Lions' Lemma, the Hilbert space
$$
	W=\{w\in L^2\left(0, T;\mathbf{V}\right); w_t\in L^1(0,T;\mathbf{V'})\}
$$
	is  compactly embedded
	in~$\mathbf{L}^2(Q)$ and we thus have
\begin{equation}\label{ccccconv-y-L2H01}
	\mathbf{y}_\alpha \to \mathbf{y} \hbox{ strongly in }\mathbf{L}^2(Q).
\end{equation}
	Also, using the second equation in \eqref{Lalpha} we see that
\[
\begin{alignedat}{2}
	( \mathbf{z}_\alpha- \mathbf{y}) - \alpha^2\Delta( \mathbf{z}_\alpha- \mathbf{y} )+\nabla\pi
	=( \mathbf{y}_\alpha- \mathbf{y})+\alpha^2\Delta  \mathbf{y}.
\end{alignedat}
\]
	Therefore, after some computations, we deduce that
\begin{equation}\label{ccccconv-z-L2}
	 \mathbf{z}_\alpha \rightarrow  \mathbf{y}\hbox{ strongly in } \mathbf{L}^2(Q).
\end{equation}
	This proves that we can find $p$ such that $(\mathbf{y},p)$ is solution to~\eqref{NS2}.

	In other words, at least for a subsequence, the solutions to the Leray-$\alpha$ system converge
	(in the sense of~\eqref{converyvec}) towards a solution to the Navier-Stokes system.

	Let us now consider the following controlled systems for the Navier-Stokes and Leray-$\alpha$ systems:
\begin{equation}\label{NS}
	\left\{
		\begin{array}{lll}
     			\mathbf{y}_t  - \Delta \mathbf{y} +(\mathbf{y}\cdot \nabla) \mathbf{y}+ \nabla p =  \mathbf{v}1_\omega  			  & \text{in}&       Q,     \\
     			\nabla \cdot \mathbf{y} =0                                                                             								  & \text{in}&       Q,     \\
     			\mathbf{y} = \textbf{0}                                                                              								          & \text{on}&  \Sigma, \\
     			\mathbf{y}(0) = \mathbf{y}_0                                                                          								  & \text{in}&  \Omega
		\end{array}
	\right.
\end{equation}
	and
\begin{equation}\label{CLalpha}
	\left\{
		\begin{array}{lll}
			\mathbf{y}_t  - \Delta \mathbf{y} +(\mathbf{z} \cdot \nabla) \mathbf{y} + \nabla p = \mathbf{v}1_\omega 		& \text{in}&       Q, \\
			\nabla \cdot \mathbf{y}= \nabla \cdot \mathbf{z}= 0                                                     					& \text{in}&       Q, \\
			\mathbf{z}-\alpha^2\Delta \mathbf{z} +\nabla \pi=\mathbf{y}                                             					& \text{in}&       Q, \\
			\mathbf{y} = \mathbf{z}= \mathbf{0}                                                                     						& \text{on}&  \Sigma, \\
			\mathbf{y}(0) = \mathbf{y}_0                                                                            							& \text{in}&  \Omega,
		\end{array}
	\right.
\end{equation}
	where $\mathbf{v} = \mathbf{v}(x, t)$ stands for the control function.

	With arguments similar to those in~\cite{FC-G-P}, it can be proved that, for any~$T > 0$, there exists $\eps > 0$ such that, if
	$\| \mathbf{y}_0 \| < \eps$, for each $\alpha > 0$ we can find controls $\mathbf{v}_\alpha \in \mathbf{L}^2(\om\times(0,T))$ and
	associate states $(\mathbf{y}_\alpha,p_\alpha,\mathbf{z}_\alpha,\pi_{\alpha})$ satisfying
\[
	\mathbf{y}_\alpha(x,T) = \mathbf{0} \quad \text{in} \quad \Om.
\]

	{
	In a forthcoming paper, we will show that these null controls $\mathbf{v}_\alpha$ can be bounded independently of $\alpha$ and a result similar to Theorem~\ref{conver-cont} holds for~\eqref{CLalpha}.	
	}


\end{document}